\documentclass[12pt]{amsart}
\usepackage{wrapfig}
\usepackage{graphicx}
\usepackage[all]{xy}
\usepackage{amsmath,amsthm,amssymb}
\usepackage{mathabx,epsfig}

\usepackage[margin=1in]{geometry}
\usepackage{amsmath,amscd}
\usepackage{relsize}
\usepackage{color}
\usepackage{mathtools}

\usepackage[all]{xy}

\usepackage{thmtools}
\usepackage{thm-restate}
\usepackage{hyperref}
\usepackage{cleveref}

\newtheorem{theorem}{Theorem}[section]
\newtheorem{conjecture}[theorem]{Conjecture}

\newtheorem{lemma}[theorem]{Lemma}
\newtheorem{proposition}[theorem]{Proposition}
\usepackage[small,compact]{titlesec}
\usepackage{enumerate}
\usepackage{enumitem}
\usepackage{times}

\makeatletter

\newcommand{\Rmnum}[1]{\expandafter\@slowromancap\romannumeral #1@}
\makeatother

  \newcommand{\PGL}{\operatorname{PGL}}

 \newcommand{\Pic}{\operatorname{Pic}}

 \newcommand{\PrePer}{\operatorname{PrePer}}

 \newcommand{\Orb}{\operatorname{Orb}}

  \newcommand{\Res}{\operatorname{Res}}
 \newcommand{\Rat}{\operatorname{Rat}}
\newcommand{\Avg}{\operatorname{Avg}}

\usepackage{amssymb,fge}
\newcommand{\mysetminus}{\mathbin{\fgebackslash}}

\usepackage{xcolor}
\usepackage{titlesec}
\titleformat{\section}[block]{\color{black}\large\filcenter}{}{1em}{}
\titleformat{\subsection}[hang]{\bfseries}{}{1em}{}
\theoremstyle{remark}
 
\newtheorem{remark}[theorem]{Remark}
\begin{document}
\title{The average number of integral points in orbits}
\author{Wade Hindes}
\dedicatory{Dedicated to Joseph H. Silverman on the occasion
     of his $60$th birthday.}
\date{\today}
\maketitle
\renewcommand{\thefootnote}{}
\begin{abstract} Over a number field $K$, a celebrated result of Silverman states that if $\varphi(z)\in K(z)$ is a rational function whose second iterate is not a polynomial, the set of $S$-integral points in the orbit $\Orb_\varphi(P)=\{\varphi^n(P)\}_{n\geq0}$ is finite for all $P\in \mathbb{P}^1(K)$. In this paper, we show that if we vary $\varphi$ and $P$ in a suitable family, the number of $S$-integral points in $\Orb_\varphi(P)$ is absolutely bounded. In particular, if we fix $\varphi$ and vary the basepoint $P\in \mathbb{P}^1(K)$, then we show that $\#(\Orb_\varphi(P)\cap\mathcal{O}_{K,S})$ is zero on average. Finally, we prove a zero-average result in general, assuming a standard height uniformity conjecture in arithmetic geometry.     
\end{abstract} 
\begin{section}{Introduction} 
Let $K/\mathbb{Q}$ be a number field and let $S$ be a finite set of places containing the archimedean ones. Let $\varphi(z)\in K(z)$ be a rational function of degree $d\geq2$, and let $\varphi^n$ denote the $n^{th}$ iterate of $\varphi$. If $\varphi^2$ is not a polynomial, then Silverman proved in \cite{Silv-Int} that the forward orbit
\begin{equation}{\label{orbitdef}} 
\Orb_\varphi(P):=\{P,\varphi(P),\varphi^2(P),\dots\}
\end{equation} 
contains only finitely many $S$-integral points for all $P\in\mathbb{P}^1(K)$. Moreover, Hsia and Silverman \cite{Silv-Hsia} gave an explicit bound on the number of $S$-integral points in $\Orb_\varphi(P)$, though it is normally much larger than the actual number. Nevertheless, there are rational maps (of every degree) with arbitrarily many integral points, illustrating the problem's subtlety \cite[Proposition 3.46]{Silv-Dyn}.  

On the other hand, one may hope to control $\#(\Orb_{\varphi}(P)\cap\mathcal{O}_{K,S})$ on average, that is, as we vary over $P\in\mathbb{P}^1(K)$. This point of view has yielded some powerful insight in other areas of number theory \cite{Avgintell, Avgrank, P-Stoll}, and we proceed with this approach here.

We begin by fixing some notation. A rational map $\varphi:\mathbb{P}^1\rightarrow\mathbb{P}^1$ of degree $d$ is defined by two homogenous polynomials 
\[\varphi=[F,G]=[a_dx^d+a_{d-1}x^{d-1}y+\dots +a_0y^d,\, b_dx^d+b_{d-1}x^{d-1}y+\dots +b_0y^d]\]
such that the resultant $\Res(F,G)\neq0$. In this way, a rational map is determined by a $(2d+2)$-tuple of numbers $(a_0,a_1,\dots,a_d,b_0,b_1,\dots b_d)$, well-defined up to scaling. In particular, we may identify the set of rational maps of degree $d$, denoted $\Rat_d$, as an open subset of $\mathbb{P}^{2d+1}$; see \cite[\S4.3]{Silv-Dyn}. Moreover, we define the height of $\varphi$, written $h(\varphi)$, to be its corresponding height in projective space \cite[\S3.1]{Silv-Dyn}.     

In this paper, we consider integral points in families $\phi:X\rightarrow\Rat_d$ of dynamical systems, where $X_{/K}$ is a projective variety and $\phi$ is a rational map defined over $K$. Specifically, if $X$ is equipped with a morphism $\beta:X\rightarrow\mathbb{P}^1$, then we study $(\Orb_{\phi_P}(\beta_P)\cap\mathcal{O}_{K,S})$, given by evaluating $\phi$ and $\beta$ at suitable points $P\in X(K)$.

\begin{remark} Let $K(X)$ denote the function field of $X$. Then we can view the family $(X,\phi,\beta)$ as given by a rational function $\phi(z)\in K(X)(z)$ of degree $d\geq2$ and a point $\beta\in K(X)$, and we will use both perspectives here.    
\end{remark} 

To make our notion of suitable points explicit, define the following subset of $X$:
\begin{equation}{\label{i}} I_{X,\phi}:=\big\{P\in X\;|\; \phi_P\in \Rat_d\;\text{is defined},\,\text{and}\; \phi_P^2\notin\widebar{K}[x]\big\}. 
\end{equation}
When $X$ is a curve, we prove that the quantity $\#(\Orb_{\phi_P}(\beta_P)\cap\mathcal{O}_{K,S})$ is uniformly bounded over all points $P\in I_{X,\phi}(K)$ and all morphisms $\beta:X\rightarrow\mathbb{P}^1$. However, as is often the case when studying families of dynamical systems, we must assume that the generic map $\phi(z)\in K(X)(z)$ is not \emph{isotrivial}: $\phi$ is isotrivial if there is a M\"{o}bius transformation $M\in\PGL_2(\overline{K(X)})$ so that $M^{-1}\circ\phi\circ M$ is defined over the constant field $\overline{K}$. 

In particular, given the freedom we have in choosing the basepoint family $\beta$, we have made progress towards a dynamical analog of a conjecture of Lang \cite[page 140]{Lang-Ellip} regarding the number of integral points on elliptic curves; see \cite[Conjecture 3.47]{Silv-Dyn}. \vspace{.2cm}      
\begin{theorem}
\label{thm:families}
Let $X_{/K}$ be a curve and let $\phi:X\rightarrow\Rat_d\subseteq\mathbb{P}^{2d+1}$ and $\beta:X\rightarrow\mathbb{P}^1$ be rational maps over $K$. If $\phi(z)\in K(X)(z)$ of degree $d\geq2$ is not isotrivial, then $\#(\Orb_{\phi_P}(\beta_P)\cap\mathcal{O}_{K,S})$ is uniformly bounded over all points $P\in I_{X,\phi}(K)$. \vspace{.2cm}  
\end{theorem} 
Of course, one expects that the actual number of integral points in a particular orbit is zero, provided that $X$ has sufficiently many rational points. To make this statement precise, we define a suitable notion of average. Given an ample height function $h_X$ on $X$ and a positive real number $B$, we write $I_{X,\phi}(B,K)$ for the set of points in $X(K)\cap I_{X,\phi}$ of height at most $B$. Then we define the \emph{average number of $S$-integral points in orbits in the family} $(X,\phi,\beta)$ to be \vspace{.2cm} 
\begin{equation}{\label{avgdef}} 
\widebar{\Avg}(\phi,\beta,S):=\limsup_{B\rightarrow \infty}\frac{\sum_{P\in I_{X,\phi}(B,K)}\#\big(\Orb_{\phi_P}(\beta_P)\cap\mathcal{O}_{K,S}\big)}{\#I_{X,\phi}(B,K)}, \vspace{.1cm}
\end{equation} 
and we show that $\widebar{\Avg}(\phi,\beta,S)=0$ for several families $(X,\phi,\beta)$ over number fields, including the case of constant families: $\phi:X\rightarrow\Rat_d$ with $\phi_P=\varphi$ for all $P\in X$. \vspace{.2cm}   
\begin{theorem}
\label{thm:2}
Let $\varphi(z)\in K(z)$ be such that $\deg(\varphi)\geq2$ and $\varphi^2(z)\notin \widebar{K}[z]$. If $X_{/K}$ is a curve of genus $g\geq1$ and $\beta:X\rightarrow\mathbb{P}^1$ is a non-constant map, then the set \vspace{.05cm}
\[\{P\in X(K)\;|\; \big(\Orb_\varphi(\beta_P)\cap\mathcal{O}_S\big)\neq\varnothing\}\vspace{.05cm}\] 
is finite. Moreover, if $g\leq1$ and $X(K)$ is infinite, then \vspace{.1cm}   
\begin{equation*} 
\widebar{\Avg}(\varphi,\beta,S):=\limsup_{B\rightarrow\infty}\frac{\sum_{P\in I_{X,\varphi}(B,K)}\#\big(\Orb_{\varphi}(\beta_P)\cap\mathcal{O}_{K,S}\big)}{\#I_{X,\varphi}(B,K)}=0\vspace{.15cm}   
\end{equation*}
In particular, if $X=\mathbb{P}^1$ and $\beta$ is the identity map, \vspace{.1cm}    
\[\widebar{\Avg}(\varphi,S):=\limsup_{B\rightarrow \infty}\frac{\sum_{P\in \mathbb{P}^1(B,K)}\#\big(\Orb_{\varphi}(P)\cap\mathcal{O}_{K,S}\big)}{\#\mathbb{P}^1(B,K)}=0.\vspace{.2cm} \] 
\end{theorem} 
For examples of non-constant families $(X,\phi,\beta)$ with $\widebar{\Avg}(\phi,\beta,S)=0$, see Proposition \ref{prop:example}. Moreover, assuming a standard height uniformity conjecture (see Conjecture \ref{conjecture} below) in arithmetic geometry, we prove an average-zero statement for all curves $X$ and all families $(X,\phi,\beta)$ as in Theorem \ref{thm:families}. In fact, for such families, Conjecture \ref{conjecture} implies that there is a largest iterate that can produce an integral point, a stronger statement than just bounding the number of such iterates: 
\begin{theorem}
\label{thm:conditional}
Let $X_{/K}$ be a curve and let $\phi:X\rightarrow\Rat_d\subseteq\mathbb{P}^{2d+1}$ and $\beta:X\rightarrow\mathbb{P}^1$ be rational maps over $K$. If $d\geq2$ and 
$(X,\phi,\beta)$ satisfies:  
\begin{enumerate}[topsep=8pt, partopsep=8pt, itemsep=10pt]  
\item[\textup{(1)}] $\phi^2(z)\not\in K(X)[z]$ (i.e. the second iterate of $\phi$ is not generically a polynomial),  
\item[\textup{(2)}] $\hat{h}_\phi(\beta)>0$,   
\end{enumerate} 
then $P\in I_{X,\phi}(K)$ and $\hat{h}_{\phi_P}(\beta_P)>0$ for all but finitely many points in $X(K)$, and Conjecture \ref{conjecture} implies that \vspace{.1cm}
\[N(\phi,\beta,S):=\sup\big\{n\,:\, \phi_P^n(\beta_P)\in\mathcal{O}_{K,S}\;\; \text{for some}\; P\in I_{X,\phi}(K),\; \hat{h}_{\phi_P}(\beta_P)>0\big\} \vspace{.05cm}\] 
is finite. Moreover, if $X(K)$ is infinite, then  
\begin{equation*} 
\widebar{\Avg}(\phi,\beta,S):=\limsup_{B\rightarrow \infty}\frac{\sum_{P\in I_{X,\phi}(B,K)}\#\big(\Orb_{\phi_P}(\beta_P)\cap\mathcal{O}_{K,S}\big)}{\#I_{X,\phi}(B,K)}=0\vspace{.1cm}.  
\end{equation*}  
\end{theorem} 
Finally, we outline how one might generalize these results to varieties of arbitrary dimension, and we illustrate this idea with an explicit $3$-dimensional family in Proposition \ref{prop:exam}.        
\end{section}  
\begin{section}{Integral points in orbits in families}
Throughout this section, let $X$ be curve, let $K$ be a number field, and let $d\geq2$. We use some properties of height functions on $X$ to prove that the quantity $\#(\Orb_{\phi_P}(\beta_P)\cap\mathcal{O}_{K,S})$ is bounded for all non-isotrivial families $(X,\phi,\beta)$. In particular, it follows that the average $\widebar{\Avg}(\phi,\beta,S)$ is also bounded for such families. To prove Theorem \ref{thm:families}, we use the following result due to Call and Silverman to estimate canonical heights; see \cite[Theorem 4.1]{Call-Silverman}.    
\begin{theorem}{\label{CallSilverman}} Let $X_{/K}$ be a curve, let $\phi:X\rightarrow\Rat_d$ and $\beta:X\rightarrow\mathbb{P}^1$ be rational maps over $K$ defined on $X_0\subseteq X$, and let $h_X$ be an ample height function on $X$. Then,  
\[\lim_{\substack{h_X(P)\rightarrow\infty\\ P\in X_0}} \frac{\hat{h}_{\phi_P}(\beta_P)}{h_X(P)}=\hat{h}_{\phi}(\beta).\] 
Here $\hat{h}_{\phi}(\beta)$ is defined by \cite[Theorem 3.20]{Silv-Dyn} using the Weil height on $K(X)$.  
\end{theorem}  
\begin{proof}[(Proof of Theorem \ref{thm:families})] As \cite[Corollary 17]{Silv-Hsia} suggests, we must bound the ratio $h_{\mathbb{P}^{2d+1}}(\phi_P)\big/\,\hat{h}_{\phi_P}(\beta_P)$ as we range over suitable points $P\in I_{X,\phi}(K)$, to bound the number of integral points in orbits. To do this, fix an ample height function $h_X$ on $X$ corresponding to a divisor of degree $1$ and let $H$ be a hyperplane class in $\Pic(\mathbb{P}^{2d+1})$. In particular, the normalization property of heights implies that 
\begin{equation}{\label{htbd1}} h_{\mathbb{P}^{2d+1},H}(Q)=h_{\mathbb{P}^{2d+1}}(Q)+O(1)
\end{equation}  
for all $Q\in\mathbb{P}^{2d+1}(\widebar{K})$; see, \cite[Theorem III.10.1(a)]{Silv-Advanced}. On the other hand, since the rational map $\phi:X\rightarrow\Rat_d$ extends to a morphism $\phi:X\rightarrow\mathbb{P}^{2d+1}$ of projective varieties \cite[II.2 Prop 2.1]{Silv-Arith}, the functorality of heights  \cite[Theorem III.10.1(d)]{Silv-Advanced} and (\ref{htbd1}) imply that 
\begin{equation}{\label{htbd2}}
h_{X,\,\phi^* H}(P)=h_{\mathbb{P}^{2d+1}}(\phi_P)+O(1)
\end{equation}
for all $P\in X(\overline{K})$. Now we compare the two height functions $h_X$ and $h_{X,\,\phi^* H}$ on the curve $X$ using \cite[Theorem III.10.2]{Silv-Advanced}: 
\begin{equation}{\label{htbd3}} \lim_{h_X(P)\rightarrow\infty}\frac{h_{X,\,\phi^* H}(P)}{h_X(P)}=\deg(\phi^*H). 
\end{equation} 
Here we use that $h_X$ was associated to a divisor of degree $1$. From here, we proceed in cases. Suppose first that $\hat{h}_\phi(\beta)>0$. Then Theorem \ref{CallSilverman}, (\ref{htbd2}) and (\ref{htbd3}) together imply that \vspace{.2cm}
\begin{equation*} 
\begin{split}
\;\lim_{\substack{h_X(P)\rightarrow\infty\\ P\in X_0}}\, \frac{h_{\mathbb{P}^{2d+1}}(\phi_P)}{\hat{h}_{\phi_P}(\beta_P)}&=\lim_{\substack{h_X(P)\rightarrow\infty\\ P\in X_0}} \,\frac{h_{X,\,\phi^* H}(P)+O(1)}{\hat{h}_{\phi_P}(\beta_P)}=\lim_{\substack{h_X(P)\rightarrow\infty\\ P\in X_0}} \,\frac{h_{X,\,\phi^* H}(P)+O(1)}{h_X(P)}\cdot \frac{h_X(P)}{\hat{h}_{\phi_P}(\beta_P)}\\[12pt]
&=\lim_{h_X(P)\rightarrow\infty} \,\frac{h_{X,\,\phi^* H}(P)+O(1)}{h_X(P)}\;\cdot\, \lim_{\substack{h_X(P)\rightarrow\infty\\ P\in X_0}}\,\frac{h_X(P)}{\hat{h}_{\phi_P}(\beta_P)}\,=\;\frac{\deg(\phi^*H)}{\hat{h}_\phi(\beta)}.
\end{split}  
\end{equation*} 
In particular, for all points $P\in X_0$ with $h_X(P)$ bigger than some fixed $\delta>0$, the ratio 
\begin{equation}{\label{htbd4}} 
\frac{h_{\mathbb{P}^{2d+1}}(\phi_P)}{\hat{h}_{\phi_P}(\beta_P)}\leq \frac{\deg(\phi^*H)}{\hat{h}_\phi(\beta)}+1
\end{equation} 
is bounded independently of $P$. In particular, (\ref{htbd4}) and \cite[Corollary 17]{Silv-Hsia} imply that there is a constant $\gamma=\gamma\big(d,[K:\mathbb{Q}]\big)$ such that \vspace{.05cm}
\begin{equation}{\label{bdint}} \#\big(\Orb_{\phi_P}(\beta_P)\cap\mathcal{O}_{K,S}\big)\leq 4^{\#S}\cdot\gamma+ \log_d^+\bigg(\frac{\deg(\phi^*H)}{\hat{h}_\phi(\beta)}+1\bigg) \vspace{.05cm}
\end{equation}
for all $P\in I_{X,\phi}(K)$ with $h_X(P)>\delta$. On the other hand, $h_X$ is ample, so that any set of $K$-points of $X$ of bounded height is finite. Hence,    
\[E_\delta:=\{P\in X(K)\,:\, h_X(P)\leq\delta\}\]
is a finite set. Therefore, if $P\in I_{X,\phi}(K)$ and $P\in E_\delta$, then $\#(\Orb_{\phi_P}(\beta_P)\cap\mathcal{O}_{K,S})$ is uniformly bounded by Silverman's original theorem; see \cite[Theorem 3.43]{Silv-Dyn} or \cite{Silv-Int} for the general statement. This fact, along with the bound in (\ref{bdint}), completes the proof of Theorem \ref{thm:families} when $\hat{h}_\phi(\beta)>0$.
 
On the other hand, if $\hat{h}_\phi(\beta)=0$, then \cite[Corollary 1.8]{Baker} implies that $\beta$ is preperiodic for $\phi$: there exist two (distinct) non-negative integers $n$ and $m$ such that $\phi^n(\beta)=\phi^m(\beta)$. In particular, this holds for every specialization, i.e. $\phi_P^n(\beta_P)=\phi_P^m(\beta_P)$ for all $P\in X$. Therefore, we see that $\#\Orb_{\phi_P}(\beta_P)\leq\max\{n,m\}$ for all $P\in X$. Hence, $\#(\Orb_{\phi_P}(\beta_P)\cap\mathcal{O}_{K,S})\leq\max\{n,m\}$ also, and we obtain a trivial bound in this case.              
\end{proof}  
\begin{remark} We note that the uniform bound in Theorem \ref{thm:families} need not hold for isotrivial families. To see this, we use Silverman's original ``clearing denominators" trick in \cite[Proposition 3.46]{Silv-Dyn}. Fix $\varphi(z)\in\mathbb{Q}(z)$ such that $\varphi^2(z)\not\in\mathbb{Q}[z]$ and $\hat{h}_\varphi(0)>0$; write $\varphi^n(0)=a_n/b_n$ for some $a_n, b_n\in\mathbb{Z}$. Moreover, set $B_N=\prod_{i=1}^N b_i$ for any $N\geq1$. Now let $X=\mathbb{P}^1$ and define the family $(X,\phi,\beta)$ given by $\phi_t(z)=t\cdot\varphi(z/t)$ and $\beta_t=0$ for all $t\in X\mysetminus\{\infty\}$. Then one checks that for the specializations $t:=B_N$, we have a lower bound $N\leq\#(\Orb_{\phi_t}(\beta_t)\cap\mathbb{Z})$ for all $N$. Hence, there is no uniform bound as in Theorem \ref{thm:families}.   
\end{remark}  
Although it is nice to have an upper bound, one expects that most orbits contain no integers, provided that $X$ has sufficiently many points. We prove this in the case of constant families over a curve; see Theorem \ref{thm:2}. To do this, we need a different sort of bound on the number of integral points in orbits than that given in \cite[Corollary 17]{Silv-Hsia}.   
\begin{lemma}{\label{lemma:2}} There exists an $N(\varphi,S)>0$ such that $\varphi^n(a)\in\mathcal{O}_{K,S}$ implies $n\leq N(\varphi,S)$ for all $\varphi$-wandering points $a\in \mathbb{P}^1(K)$. 
\end{lemma} 
\begin{proof}[(Proof of Lemma \ref{lemma:2})] Suppose that $\varphi^n(a)\in \mathcal{O}_{K,S}$ and that $n\geq4$. Since, $\varphi^2(z)\notin \widebar{K}[z]$, it follows from the Riemann-Hurwitz formula that $\#\varphi^{-4}(\infty)\geq 3$; see \cite[Proposition 3.44]{Silv-Dyn}. In particular, the set of $S$-integral preimages  
\begin{equation} T_4(\varphi,S):=\{b\in \mathbb{P}^1(K)\;|\; \varphi^4(b)\in\mathcal{O}_{K,S}\} 
\end{equation}  
is finite; see \cite[Theorem 3.36]{Silv-Dyn}. Note that $\varphi^{4}(\varphi^{n-4}(a))=\varphi^n(a)\in\mathcal{O}_{K,S}$ and $\varphi^{n-4}(a)\in T_4(\varphi,S)$. Hence, $h(\varphi^{n-4}(a))$ is bounded independently of both $a$ and $n$. So together with part (b) and (c) of \cite[Proposition 6]{Silv-Hsia}, we see that $\deg(\varphi)^{n-4}\cdot \hat{h}_\varphi(a)=\hat{h}_\varphi(\phi^{n-4}(a))$ is bounded. Moreover, 
\begin{equation}{\label{hatmin}} \hat{h}^{\min}_{\varphi,K}:=\inf\{\hat{h}_\varphi(c)\;|\: c\in\mathbb{P}^1(K)\; \text{wandering for}\; \varphi\}
\end{equation} is strictly positive. To see this, choose an arbitrary wandering point $c_0\in K$ for $\varphi$ (possible, for instance, by Northcott's Theorem \cite[Theorem. 3.12]{Silv-Dyn}), and note that 
\[\hat{h}^{\min}_{\varphi,K}=\inf\{\hat{h}_\varphi(c)\;|\: c\in\mathbb{P}^1(K)\;\text{and}\; 0<\hat{h}_\varphi(c)<\hat{h}_\varphi(c_0)\}.\] However, this latter set is finite and consists of strictly positive numbers; hence $\hat{h}^{\min}_{\varphi,K}> 0$. Putting this together with the fact that $\deg(\varphi)^{n-4}\cdot \hat{h}^{\min}_{\varphi,K}\leq\deg(\varphi)^{n-4}\cdot \hat{h}_\varphi(a)$ is bounded by the height of points in $T_4(\varphi,S)$, we see that $n$ is bounded independently of $a$ as desired.           
\end{proof}

\begin{proof}[(Proof of Theorem \ref{thm:2})]Since $\PrePer(\varphi,K)$ is finite \cite[Theorem. 3.12]{Silv-Dyn} and $\beta:X\rightarrow\mathbb{P}^1$ is non-constant, it follows that 
\begin{equation}{\label{PrePer}} X_{\varphi,\beta}^{\PrePer}(K):=\{P\in X(K)\;|\;\beta_P\in\PrePer(\varphi,K)\}
\end{equation} 
is finite. In particular, for both statements of Theorem \ref{thm:2}, it suffices to assume that $P\in X(K)$ is such that $\beta_P$ is a wandering point of $\varphi$. In light of Lemma \ref{lemma:2}, we define the set  
\begin{equation} T_n(\varphi,\beta,S):=\{P\in X(K)\;|\; \varphi^n(\beta_P)\in\mathcal{O}_{K,S},\;\hat{h}_{\varphi}(\beta_P)>0\}.
\end{equation}    
Suppose that $X$ has genus $g\geq1$. In this case, it follows from a theorem of Siegel that $T_n(\varphi,\beta,S)$ is finite for all $n\geq0$; see, for instance, \cite[Corollary IX 4.3.1]{Silv-Arith}. Moreover, Lemma \ref{lemma:2} implies that $T_n(\varphi,\beta,S)=\varnothing$ for all $n>N(\varphi,S)$. Hence, \[\{P\in X(K)\;|\; \big(\Orb_\varphi(\beta_P)\cap\mathcal{O}_{K,S}\big)\neq\varnothing\}\]     
is finite in the positive genus case as claimed. 

On the other hand, when $X$ is a rational curve ($g=0$), we may assume that $X=\mathbb{P}^1$. In this case, $T_n(\varphi,\beta,S)$ can be infinite; see Remark \ref{rmk1} below. However, we will show that $T_n(\varphi,\beta,S)$ is sparse in $\mathbb{P}^1(K)$. With this in mind, for any subset $T\subseteq \mathbb{P}^1(K)$, define the \emph{upper density} of $T$ to be the quantity 
\begin{equation}{\label{density}} \widebar{\delta}_{K}(T):=\limsup_{B\rightarrow\infty}\frac{\sum_{\{P\in T\;|\; H_{\mathbb{P}^1}(P)\leq B\}}1}{\sum_{\{P\in \mathbb{P}^1(K)\;|\; H_{\mathbb{P}^1}(P)\leq B\}}1}. 
\end{equation}
Note that since $X_{\varphi,\beta}^{\PrePer}(K)$ on (\ref{PrePer}) is finite, Lemma \ref{lemma:2} implies that
\begin{equation}{\label{estimate}} \widebar{\Avg}(\varphi,\beta,S)\leq N(\varphi,S)\,\sum_{n=0}^{N(\varphi,S)}\widebar{\delta}_{K}\big(T_n(\varphi,\beta,S)\big).
\end{equation}  
In particular, it suffices to prove that $\widebar{\delta}_{K}(T_n(\varphi,\beta,S))=0$ for all $1\leq n\leq N(\varphi,S)$. Letting $f=\varphi^n\circ\beta$, this follows from the following lemma. 
\end{proof} 
\begin{lemma}{\label{lemma:density}} Let $f(z)\in K(z)$ be a non-constant rational function and let 
\[ T(f,S):=\{P\in\mathbb{P}^1(K)\;|\; f(P)\in\mathcal{O}_{K,S}\}\]
be the set of points in $\mathbb{P}^1(K)$ with $S$-integral images under $f$. Then 
\begin{equation*}
\widebar{\delta}_{K}(T(f,S)):=\limsup_{B\rightarrow\infty}\frac{\sum_{\{P\in T(f,S)\;|\; H_{\mathbb{P}^1}(P)\leq B\}}1}{\sum_{\{P\in \mathbb{P}^1(K)\;|\; H_{\mathbb{P}^1}(P)\leq B\}}1}=0.  
\end{equation*} 
That is, $T(f,S)$ has upper density zero in $\mathbb{P}^1(K)$. 
\end{lemma} 
\begin{proof}[Proof of Lemma \ref{lemma:density}] 
Note that $T(f,S)\leq T(f,S')$ whenever $S\subseteq S'$. Therefore, we may enlarge $S$ and assume that $f$ has good reduction outside of $S$. To count elements of $T(f,S,B)$, we sieve out points of $\mathbb{P}^1(K)$ given by local congruence conditions. To do this, let $\mathcal{P}$ be the set of primes $\mathfrak{p}\subseteq\mathcal{O}_{K}$, disjoint from $S$, such that $f$ has a pole $a_{\mathfrak{p}}\in\mathbb{P}^1(\mathbb{F}_\mathfrak{p})$; here, $\mathbb{F}_\mathfrak{p}$ is the residue field at $\mathfrak{p}$. In particular, it follows from the Chebotarev density theorem that $\mathcal{P}$ has positive Dirichlet density. Let $\pi_{\mathfrak{p}}:\mathbb{P}^1(K_\mathfrak{p})\rightarrow \mathbb{P}^1(\mathbb{F}_\mathfrak{p})$ be the reduction map, and consider the set
\[\mathcal{I}_\mathfrak{p}:=\{P\in\mathbb{P}^1(K):\,\pi_\mathfrak{p}(P)=a_\mathfrak{p}\}.\]
It follows from the proof of Schanuel's Theorem \cite[Corollary 1 p.447]{Schanuel} that $\mathcal{I}_\mathfrak{p}$ has density $(N(\mathfrak{p})+1)^{-1}$ in $\mathbb{P}^1(K)$. In other words, the residue classes modulo $\mathfrak{p}$ equidistribute with respect to the Weil height.  To see this, simply replace the lattice $\Lambda\subseteq\mathbb{R}^k$ with a translate $\alpha +\Lambda$ for $\alpha\in\mathbb{R}^k$ in \cite[Theorem 2]{Schanuel}, and deduce that $\overline{\delta}_K$ in (\ref{density}) is translation invariant. In particular, by choosing lifts (in $K$) for each residue class in $\mathbb{P}^1(\mathbb{F}_\mathfrak{p})$ and translating: for instance applying a map of the form $P\rightarrow P+(b-a)$, we see that 
\[\overline{\delta}_K\Big(\{P\in\mathbb{P}^1(K)\,:\, \pi_\frak{p}(P)=\bar{a}\}\Big)=\overline{\delta}_K\Big(\{P\in\mathbb{P}^1(K)\,:\, \pi_\frak{p}(P)=\bar{b}\}\Big)\]  
for all $\bar{a},\bar{b}\in\mathbb{P}^1(\mathbb{F}_\frak{p})$. In particular, since distinct residue classes are disjoint and their union covers $\mathbb{P}^1(K)$, it follows that $\overline{\delta}_K(\mathcal{I}_\mathfrak{p})=(1+N(\mathfrak{p}))^{-1}$ as claimed. Moreover, translation invariance and the Chinese Remainder Theorem imply that      
\[\overline{\delta}_K\bigg(\bigcap_{N(\mathfrak{p})\leq x}\mathcal{I}_\mathfrak{p}\bigg)=\prod_{N(\mathfrak{p})\leq x}\frac{1}{1+N(\mathfrak{p})}\]
for all $x\in\mathbb{R}_{\geq0}$. On the other hand, $T(f,S)$ is disjoint from $\mathcal{I}_\mathfrak{p}$ for all $\mathfrak{p}\in\mathcal{P}$: if $P\in \mathcal{I}_\mathfrak{p}$ then $\pi_\frak{p}(f(P))=f(a_\frak{p})=\infty_\mathfrak{p}$, since $f$ has good reduction at $\mathfrak{p}$ and $\mathfrak{p}$ is not in $S$. In particular, the inclusion-exclusion principle implies that    
\begin{equation}
\begin{split}
\overline{\delta}_K(T(f,S))\leq\prod_{N(\mathfrak{p})\leq x}\bigg(1-\frac{1}{1+N(\mathfrak{p})}\bigg). 
\end{split} 
\end{equation} 
However, we can let $x$ grow.  Since $\mathcal{P}$ had positive density in the primes, the product above converges to $0$ as $x \to \infty$: recall that an infinite product $\prod (1 - a_i)$, with $0 \leq a_i < 1$, converges to $0$ if and only if $\sum a_i$ diverges. In our case, $a_i$ is on the order of $N(\mathfrak{p}_i)^{-1}$, and since the sum of reciprocals of a set  of primes of positive density diverges, the infinite product converges to $0$. 
\end{proof} 
\begin{remark}{\label{rmk1}} When $X=\mathbb{P}^1$, it is possible that $T_n(\varphi,\beta,S)$ is infinite, even if one assumes that $\varphi^2(z)\notin K[z]$. For example, let $F(z)\in\mathbb{Z}[z]$ be any polynomial of degree $2d$, let $D>1$ be any square-free integer, and let   
\begin{equation} 
\varphi(z)=\frac{F(z)}{(z^2-D)^d}.
\end{equation} 
 If $(u,v)\in\mathbb{Z}^2$ is a solution to the Pell equation $u^2-Dv^2=1$, then $\varphi(u/v)=v^{2d}\cdot F(u/v)\in\mathbb{Z}$. Setting $\beta(t)=t$, we see that $T_n(\varphi,\beta,S)$ is infinite. However, the set of coprime pairs $(u,v)$ satisfying the Pell equation is sparse in $\mathbb{P}^1(\mathbb{Q})$.          
\end{remark} 
We would like to extend Theorem \ref{thm:2} to non-constant families of rational maps. To do this, note that if $(X,\phi,\beta)$ is a family as in Theorem \ref{thm:families}, then the average number of $S$-integral points in $\Orb_{\phi_P}(\beta_P)$ is bounded. In particular, such families are a good place to test generalizations of Theorem \ref{thm:2}. In order to distill the additional properties needed, we study the following family: 

\begin{proposition} 
\label{prop:example} Let $\phi:\mathbb{P}^1\rightarrow\Rat_3$\, be the family of rational functions given by 
\begin{equation}{\label{family}} \phi_t(z):=\frac{z-t}{z^3+1}, \;\;\;\;\text{for all}\;\; t\in\mathbb{P}^1(\mathbb{Q})\mysetminus\{-1,\infty\}.  
\end{equation}
If $\beta\in\mathbb{Q}(t)$ and $\deg(\beta)\geq3$, then $\widebar{\Avg}(\phi,\beta,\mathbb{Z})=0$.  
\end{proposition} 
\begin{proof} 
We must first show that the second iterate of $\phi_t$ is not a polynomial for all $t\neq-1$. To do this, we compute that  
\begin{equation} \phi_t^2(z):=\frac{f_t(z)}{g_t(z)}=\frac{-tz^9 + z^7 - 4tz^6 + 2z^4 - 5tz^3 + z - 2t}{z^9 + 3z^6 + 4z^3 -
    3tz^2 + 3t^2z - t^3 + 1}\vspace{.05cm} 
\end{equation}
and calculate the resultant $\Res_x(f_t,g_t)=(t+1)^{12}\cdot(t^2-t+1)^{12}$. One checks that the only rational root of the resultant is $t=-1$. In particular, it follows that if $t\neq -1$, then $\phi_t^2$ is not a polynomial. From here, we show the existence of a largest iterate that can produce an integer point. To do this, we follow the outline of the proof of Siegel's theorem on integral points in \cite[Thoerem 3.36]{Silv-Dyn}. Write $t=a/b$ for some coprime $a,b\in\mathbb{Z}$ and suppose that $\phi_t^n(\beta_t)=\phi_t(\phi_t^{n-1}(\beta_t))\in\mathbb{Z}$. If we write 
\[\phi_t([x,y])=\big[bxy^2-ay^3,b(x^3+y^3)\big]\] in terms of coordinates on $\mathbb{P}^1$, then the proof of \cite[Thoerem 3.36]{Silv-Dyn} implies that 
\[\bigg\{\frac{x}{y}\in\mathbb{Q}\;\bigg|\;\phi_t\Big(\frac{x}{y}\Big)\in\mathbb{Z}\bigg\}\subseteq\bigcup_{r|(a^3+b^3)} \big\{x,y\in\mathbb{Z}\;\big|\; x^3+y^3=r\big\}. \]
Note that $|r|\leq2H(t)^3$, and combined with Lemma \ref{lemma:eg:htbd} below, we see that 
\begin{equation}{\label{c}}
H(x/y)\leq 2\sqrt{2}\cdot H(t)^{\frac{3}{2}},\,\;\;\text{whenever}\;\, \phi_t\bigg(\frac{x}{y}\bigg)\in\mathbb{Z}.
\end{equation}
In particular, we obtain the upper bound $h(\phi_t^{n-1}(\beta_t))\leq3/2\cdot h(t)+\log(2\sqrt{2})$. Moreover, parts (b) and (c) of \cite[Proposition 6]{Silv-Hsia} imply that there is a constant $c_3$ such that  \vspace{.05cm}
\begin{equation}{\label{d}}
\begin{split}
3^{n-1}\cdot\hat{h}_{\phi_t}(\beta_t)=\hat{h}_{\phi_t}(\phi_t^{n-1}(\beta_t))&\leq h(\phi_t^{n-1}(\beta_t))+5/2\cdot h(\phi_t)+c_{3}/2\\[5pt]
&\leq3/2\cdot h(t)+\log(2\sqrt{2})+ 5/2\cdot h(\phi_t)+c_{3}/2. 
\end{split} \vspace{.05cm}
\end{equation} 
On the other hand, $h(\phi_t)=h([1,-t,1,1])=h(t)$, so that (\ref{d}) implies that 
\[3^{n-1}\cdot\hat{h}_{\phi_t}(\beta_t)\leq4\cdot h(t)+\log(2\sqrt{2})+c_{3}/2.\]
Finally, parts (a) and (b) of \cite[Proposition 6]{Silv-Hsia} give the lower bound
\begin{equation}{\label{e}} (\deg(\beta)-5/2)\cdot h(t)-B_1-c_{3}/2\leq\hat{h}_{\phi_t}(\beta_t).   
\end{equation}
However, $\deg(\beta)\geq3>5/2$, and we deduce that \vspace{.05cm}
\[3^{n-1}\leq\frac{4\cdot h(t)+\log(2\sqrt{2})+c_{3}/2}{(\deg(\beta)-5/2)\cdot h(t)-B_1-c_{3}/2}, \vspace{.05cm}\]
Hence, $\phi_t^n(\beta_t)\in\mathbb{Z}$ implies that $n$ is bounded independently of $t$ (for all but finitely many $t$). On the other hand, since $(\Orb_{\phi_t}(\beta_t)\cap\mathbb{Z})$ is finite for all $t\neq-1$ by Silverman's theorem, we can bound the number of integral points in orbits for these exceptional $t$ separately. In any case, we conclude that there is an integer $N(\phi,\beta)$, such that $\phi_t^n(\beta_t)\in\mathbb{Z}$ implies $n\leq N(\phi,\beta)$ for all $t\neq-1$ (compare to Lemma \ref{lemma:2} above).  

Finally, Lemma \ref{lemma:density} implies that $T(\phi^n\circ\beta,\mathbb{Z})$ has density zero in $\mathbb{P}^1(\mathbb{Q})$ for all $n\leq N(\phi,\beta)$. It follows that $\widebar{\Avg}(\phi,\beta,\mathbb{Z})=0$ as claimed.
\end{proof}           
\begin{lemma}{\label{lemma:eg:htbd}} Suppose that $x^3+y^3=B$, for some integers $x,y, B\in\mathbb{Z}$ with $B\neq0$. Then 
\begin{equation}{\label{eg:htbd}} \max\big\{|x|,|y|\big\}\leq 2 \sqrt{|B|}. 
\end{equation}   
\end{lemma}
\begin{proof}[(Proof of Lemma \ref{lemma:eg:htbd})] We factor $x^3+y^3$ in $\mathbb{Q}[x,y]$, and write
\[ B=x^3+y^3=(x+y)\cdot(x^2-xy+y^2)=(x+y)\cdot\bigg(\frac{3}{4}(x-y)^2+\frac{1}{4}(x+y)^2\bigg).\] 
In particular, we see that 
\begin{equation}{\label{a}} \max\Big\{3/4\cdot (x-y)^2,\;1/4\cdot(x+y)^2\Big\}\leq|B|, 
\end{equation} 
since $|x+y|\geq1$ and both terms on the left side of (\ref{a}) are positive. On the other hand, it is straightforward to verify that \begin{equation}{\label{b}} \max\big\{|x|,|y|\big\}\leq\max\Big\{|x-y|,|x+y|\Big\},
\end{equation} 
and we deduce from (\ref{a}) and (\ref{b}) that $\max\{|x|,|y|\}\leq 2\sqrt{|B|}$ as claimed.      
\end{proof} 
\end{section}

\begin{section}{Height uniformity conjectures and averages in families}{\label{s3}}
The main technique we used to establish an average-zero statement for the one-dimensional families in Proposition \ref{prop:example} or the constant families in Theorem \ref{thm:2} was to prove the existence of a uniform largest iterate $N(\phi,\beta,S)$ that could produce an $S$-integral point; see Lemma \ref{lemma:2}. To find such an $N(\phi,\beta,S)$, we used Theorem \ref{CallSilverman} or \cite[Proposition 6]{Silv-Hsia} to estimate $\hat{h}_{\phi_P}(\beta_P)$, and then we bounded the height of points $Q\in\mathbb{P}^1(K)$ such that $\phi_P^4(Q)\in\mathcal{O}_{K,S}$; see Lemma \ref{lemma:eg:htbd}.

For curves $X$ and non-isotrivial families $(X,\phi,\beta)$, bounding $\hat{h}_{\phi_P}(\beta_P)$ is not a problem. On the other hand, the upper bound on the height of points $Q\in\mathbb{P}^1(K)$ such that $\phi_P^4(Q)\in\mathcal{O}_{K,S}$ follows from several height-uniformity conjectures in arithmetic geometry, including the following: 

\begin{conjecture}{\label{conjecture}} Let $\mathcal{C}\rightarrow \mathcal{B}$ be a family of curves equipped with a family of non-constant maps $\mathcal{F}\in K(\mathcal{C})$. Then there are constants $\kappa_1$ and $\kappa_2$ such that
\begin{equation}{\label{linear}} h_\mathcal{C}(Q)\leq\kappa_1\cdot h_{\mathcal{B}}(P)+\kappa_2 \;\;\;\;\;\, \text{for all} \,\;\;\;\big\{Q\in\mathcal{C}_P(K)\;\big\vert\;\, \mathcal{F}_P(Q)\in\mathcal{O}_{K,S}\big\},
\end{equation}
whenever $\mathcal{C}_P$ is a smooth curve of positive genus, or the map $\mathcal{F}_P:\mathcal{C}_P\rightarrow\mathbb{P}^1$ has at least three distinct poles. Here, $h_\mathcal{C}$ is an arbitrary height function and $h_\mathcal{B}$ is ample.     
\end{conjecture}
\begin{remark} On each of the relevant fibers, Siegel \cite{Siegel} proved that $\{Q\in\mathcal{C}_P(K)\;\vert\;\, \mathcal{F}_P(Q)\in\mathcal{O}_{K,S}\big\}$ is a finite set. Hence, Conjecture \ref{conjecture} roughly states that the heights of these points are controlled by the height of the fiber.    
\end{remark} 
\begin{remark} For families of elliptic curves, versions of Conjecture \ref{conjecture} were made by Hall and Lang \cite[IV.7]{Silv-Arith}. Moreover, Conjecture \ref{conjecture} is a consequence of the Vojta conjecture \cite[\S3.4.3]{Vojta} when the fibers $\mathcal{C}_P$ have positive genus; for justification, see \cite[Theorem 1.0.1]{Ih2} for the case of elliptic curves and \cite{Heightuniformity} or \cite[Conjecture 4]{Stoll} for the case of higher genus. Over function fields, bounds such as those on (\ref{linear}) have appeared in \cite{Kim,Schmidt}.       
\end{remark} 
Assuming Conjecture \ref{conjecture}, we prove an average-zero statement for all one-parameter families, analogous to Theorem \ref{thm:2} and Proposition \ref{prop:example} above. 
\begin{proof}[(Proof of Theorem \ref{thm:conditional})] Write $\phi^2(z)=f(z)/g(z)$ for some coprime polynomials $f,g\in K(X)[z]$. Note that $\phi^2(z)\not\in K(X)[z]$ implies that $\deg(g)\geq1$ and that $\Res(f,g)\in K(X)$ is non-zero. In particular,
\begin{equation}\label{open}  
\big\{P\in X\,:\, \deg(f_P)=\deg(f),\,\deg(g_P)=\deg(g),\, \Res(f,g)_P\neq0\big\}\subseteq I_{X,\phi}
\end{equation} 
are open subsets of $X$; here we use that taking resultants commutes with specialization, i.e. $\Res(f_P,g_P)=\Res(f,g)_P$, whenever $P\in X$ satisfies $\deg(f_P)=\deg(f)$ and $\deg(g_P)=\deg(g)$; see \cite[IV.\S8]{LangAlgebra}. 

Now, suppose that $P\in I_{X,\phi}$. Then \cite[Proposition 3.44]{Silv-Dyn} implies that $\phi_P^4\in K(z)$ has at least $3$ distinct poles. Therefore, Conjecture \ref{conjecture} applied to the (trivial) fibered surface $\mathbb{P}^1\times X \rightarrow X$ and the map $\mathcal{F}(Q,P)=\phi_p^4(Q)$, implies that there are constants $\kappa_1$ and $\kappa_2$ such that for all $Q\in\mathbb{P}^1(K)$:  
\[\mathcal{F}_P(Q)=\phi_P^4(Q)\in\mathcal{O}_{K,S}\;\;\,\text{implies}\;\;\, h(Q)\leq \kappa_1\cdot h_X(P)+\kappa_2.\]
Here $h$ is the height function on $\mathbb{P}^1\times X$ given by $h(Q,P):=h_\mathbb{P}^1(Q)$. In particular, if $P\in I_{X,\phi}(K)$ and $\phi_P^n(\beta_P)\in\mathcal{O}_{K,S}$ for some $n\geq4$, then 
\[\mathcal{F}_P(\phi_P^{n-4}(\beta_P))=\phi_P^4(\phi_P^{n-4}(\beta_P))=\phi_P^{n}(\beta_P)\in\mathcal{O}_{K,S}.\]
Hence, $h(\phi_P^{n-4}(\beta_P))\leq \kappa_1 h_X(P)+\kappa_2$. On the other hand, there are constants $\kappa_3$ and $\kappa_4$ such that: 
\[ \big|\hat{h}_{\phi_P}(x)-h(x)\big|\leq \kappa_3\cdot h_X(P)+\kappa_4, \;\;\text{for all }\,x\in \mathbb{P}^1(\overline{K});\] 
this fact follows from \cite[Theorem 3.1]{Call-Silverman} or \cite[Proposiion 6(b)]{Silv-Hsia} followed by height bounds (\ref{htbd1}), (\ref{htbd2}) and (\ref{htbd3}) above. Therefore, 
\begin{equation}{\label{conhtbd1}}
\begin{split} 
d^{n-4}\cdot \hat{h}_{\phi_P}(\beta_P)&=\hat{h}_{\phi_P}(\phi_P^{n-4}(\beta_P))\\[3pt]
&\leq h(\phi_P^{n-4}(\beta_P))+\kappa_3\cdot h_X(P)+\kappa_4\\[3pt]
&\leq (\kappa_1+\kappa_3)\,h_X(P)+(\kappa_2+\kappa_4).  
\end{split} 
\end{equation}  
Now, as in the proof of Theorem \ref{thm:families}, we use the estimate for $\hat{h}_{\phi_P}(\beta_P)$ due to Call and Silverman. Specifically, (\ref{conhtbd1}) implies that 
\[d^{n-4}\cdot \frac{\hat{h}_{\phi_P}(\beta_P)}{h_X(P)}\leq \frac{(\kappa_1+\kappa_3)\,h_X(P)+(\kappa_2+\kappa_4)}{h_X(P)}.\]  
Cleary, if $I_{X,\phi}(K)$ is finite, then there is nothing to prove. Therefore, we may assume that $I_{X,\phi}(K)$ is infinite. Hence, Theorem \ref{CallSilverman} implies that  
\begin{equation}{\label{conhtbd2}}
\begin{split}
d^{n-4}\,\hat{h}_\phi(\beta)&=\lim_{\substack{h_X(P)\rightarrow\infty\\ P\in I_{X,\phi}(K)}}d^{n-4}\; \frac{\hat{h}_{\phi_P}(\beta_P)}{h_X(P)}\\[4pt] 
&\leq\lim_{\substack{h_X(P)\rightarrow\infty\\ P\in I_{X,\phi}(K)}} \frac{(\kappa_1+\kappa_3)\, h_X(P)+(\kappa_2+\kappa_4)}{h_X(P)}=\kappa_1+\kappa_3. 
\end{split}
\end{equation}  
Consequently, for all points $P\in I_{X,\phi}(K)$ with $h_X(P)$ bigger than some fixed $\delta>0$, 
\[\phi_P^n(\beta_P)\in\mathcal{O}_{K,S}\;\;\,\text{implies}\;\;\, n\leq\max\bigg\{4,\log_d\bigg(\frac{\kappa_1+\kappa_3}{\hat{h}_\phi(\beta)}\bigg)+4\bigg\},\]
On the other hand, since $h_X$ is ample, the set of points $E_\delta:=\{P\in I_{\phi,X}(K)\,:\, h_X(P)\leq \delta\}$ is finite. Therefore, we can bound the $n$ such that $\phi_P^n(\beta_P)\in\mathcal{O}_{K,S}$ for $P\in E_\delta$ separately. In particular, we have shown that 
\begin{equation}{\label{conhtbd2}}
N(\phi,\beta,S):=\sup\big\{n\,:\, \phi_P^n(\beta_P)\in\mathcal{O}_{K,S}\;\; \text{for some}\; P\in I_{X,\phi}(K),\; \hat{h}_{\phi_P}(\beta_P)>0\big\} \vspace{.05cm} 
\end{equation} 
is bounded as claimed. The averages claim now follows from the argument given in the proof of Theorem \ref{thm:2} verbatim. Namely, if $X$ has positive genus, then the set of points in $I_{X,\phi}(K)$ such that $\phi_P^n(\beta_P)\in\mathcal{O}_{K,S}$ for some $n\leq N(\phi,\beta,S)$ is a union of finite sets, hence finite. In particular, the numerator defining $\widebar{\Avg}(\phi,b,S)$ is zero for all points of sufficiently large height. 

On the other hand, suppose that $X$ has genus $0$. Then $X\cong\mathbb{P}^1$, since $X(K)$ is non-empty. Moreover, $\mathbb{P}^1(K)\,\mysetminus\, I_{\mathbb{P}^1,\phi}(K)$ is finite by (\ref{open}). Therefore, the asymptotic count for points of bounded height in $I_{X,\phi}(K)$ is the same as that for $\mathbb{P}^1(K)$. Hence, Lemma \ref{lemma:density} and the existence of $N(\phi,\beta,S)$ implies that $\widebar{\Avg}(\phi,b,S)=0$ as in the proof of Theorem \ref{thm:2}.                    
\end{proof}
To extend Theorem \ref{thm:conditional} to non-constant families of rational maps $(X,\phi,\beta)$ parametrized by varieties of arbitrary dimension, we translate our strategy for curves to more general language. In particular, our goal (loosely speaking) is to show that the set  
\begin{equation}{\label{strategy}}
\big\{P\in I_{X,\phi}(K)\;:\; \big(\Orb_{\phi_P}(\beta_P)\cap\mathcal{O}_{K,S}\big)\neq\varnothing\big\}
\end{equation} 
is thin in $X$; see \cite[\S3.1]{Serre} for the definition of thin. From here, when $\dim(X)\geq2$, if $X$ has sufficiently many rational points and the thin subset (possibly a subvariety) containing (\ref{strategy}) is small enough, then one expects that $\widebar{\Avg}(\phi,b,S)=0$. 

On the other hand, to do this, one likely needs to control $\hat{h}_{\phi_P}(\beta_P)$ for most points of $X$. Unfortunately, this task is difficult in general. For instance, the Zariski closure of the set of points $P\in X(\widebar{K})$ such that $\hat{h}_{\phi_P}(\beta_P)=0$ has positive dimension: it contains the codimension-1 subvariety of points satisfying $\phi_P(\beta_P)=\beta_P$. 

However, a nice lower bound on canonical heights in terms of the height of the corresponding map it is predicted by \cite[Conjecture 4.98]{Silv-Dyn}, a dynamical analog of Lang's canonical height conjecture for elliptic curves \cite[page 92]{Lang-Ellip}, and it is possible that one can use this conjecture and ideas in \cite[\S3]{Call-Silverman} to attack (\ref{strategy}). Likewise when $\dim(X)\geq2$, one might try to use the Vojta conjecture, in place of Conjecture \ref{conjecture}, to prove the existence of a ``largest iterate" $N(\phi,\beta,S)$ as in Theorem \ref{thm:conditional}. 

As motivation for the study of integral points in orbits in families $(X,\beta,S)$ when $\dim(X)\geq2$, we conclude with the following example.   
\begin{proposition} 
\label{prop:exam}
Let $\phi:\mathbb{A}^1\times\mathbb{A}^1\times\mathbb{A}^1\rightarrow\Rat_3$ be the family of rational maps: 
\begin{equation}{\label{3dimeg}} \phi_{(r,s,t)}(z):=\frac{(r\cdot s)\cdot z^3+s\cdot z+t}{z^2+1}\;\;\;\;\;\text{for}\,\;r,s,t\in\mathbb{Z}. \vspace{.05cm}
\end{equation} 
If $\beta:\mathbb{A}^1\times\mathbb{A}^1\times\mathbb{A}^1\rightarrow \mathbb{A}^1$ is the map $\beta_{(r,s,t)}=r^{n_1}\cdot s^{n_2}\cdot t^{n_3}$\,\, with \,\,$\min\{n_1,n_2,n_3\}\geq6$, then \vspace{.05cm}
\begin{equation}{\label{averagedef}} \widebar{\Avg}_\mathbb{Z}(\phi,\beta,\mathbb{Z}):=\limsup_{B\rightarrow \infty}\frac{\sum_{|r|,|s|,|t|\leq B}\#\Big(\mathcal{O}_{\phi_{(r,s,t)}}\big(\beta_{(r,s,t)}\big)\cap\mathbb{Z}\Big)}{(2B+1)^3}\,<\, \infty 
\end{equation} 
\end{proposition}
A priori, it is not clear that $\Orb_{\phi_{(r,s,t)}}\big(\beta_{(r,s,t)}\big)$ contains only finitely many integers for $r,s,t\in\mathbb{Z}$; for instance, $\phi_{(1,s,0)}(z)=sz$ is a polynomial, and \cite[Theorem 3.43]{Silv-Dyn} does not apply. However, the basepoint $\beta_{(1,s,0)}=0$ is fixed in this case, and so finiteness is not a problem. Before we begin the proof of Proposition \ref{prop:exam}, we prove the stronger statement: that $\Orb_{\phi_{(r,s,t)}}\big(\beta_{(r,s,t)}\big)$ contains only finitely many integers for all rational values of $r,s,t\in\mathbb{Q}$.  
\begin{lemma}{\label{3dim}} The orbit $\Orb_{\phi_{(r,s,t)}}\big(\beta_{(r,s,t)}\big)$ contains only finitely many integers for all $r,s,t\in\mathbb{Q}$.  
\end{lemma}
\begin{proof}[(Proof of Lemma \ref{3dim})] Note that if $t=0$, then $\beta_{(r,s,t)}=0$ and $\phi_{(r,s,t)}(0)=0$. In particular, $\Orb_{\phi_{(r,s,t)}}\big(\beta_{(r,s,t)}\big)=\{0\},$ and there is nothing to prove in this case. Therefore, without loss of generality, we may assume that $t\neq0$. Likewise, if $r\cdot s=0$, then $\phi_{(r,s,t)}$ is a bounded function on the real line. In particular, the set of integers of $\Orb_{\phi_{(r,s,t)}}\big(\beta_{(r,s,t)}\big)$ is a finite. 

In the remaining cases, it suffices to show that $\phi_{(r,s,t)}^2(z)\not\in\mathbb{Q}[z]$; see \cite[Theorem 3.43]{Silv-Dyn}. We compute that \vspace{.1cm} 
\[\phi_{(r,s,t)}^2(z):=\frac{f_{(r,s,t)}(z)}{g_{(r,s,t)}(z)}:=\frac{(r^4s^4)z^9 + (3r^3s^4 + rs^2)z^7 + (3r^3s^3t + t)z^6  + \dots + (rst^3 + st + t)}{(r^2s^2)z^8 + (r^2s^2 + 2rs^2 + 1)z^6 + (2rst)z^5 + \dots + (t^2 + 1)}. \vspace{.15cm}\] 
Therefore, if $\phi_{(r,s,t)}^2(z)\in\mathbb{Q}[z]$, then $g_{(r,s,t)}(z)\cdot(az+b)=f_{(r,s,t)}(z)$ for some $a,b\in\mathbb{Q}$. By equating the $z^9$ coefficient, we see that $a:=(st)^2$. Moreover, after substituting $a:=(st)^2$ and examining the $z^8$ coefficient, we see that $b=0$. Similarly, since, $b=0$ and $t\neq0$, the $z^6$ coefficient implies that $(rs)^3=-1$, and we deduce that $(rs)=-1$. In particular, the $z^7$ coefficient implies that $s=-1$. Finally, substituting $s=-1$ in the constant term, we see that $t=0$, a contradiction.              
\end{proof} 
\begin{proof}[(Proof of Proposition \ref{prop:exam})] We begin by establishing a lower bound on $\hat{h}_{\phi_{(r,s,t)}}(\beta_{(r,s,t)})$ for all integers points $(r,s,t)$ in an open subset of $X:=\mathbb{A}^1\times\mathbb{A}^1\times\mathbb{A}^1$. In particular, we see that if $r,s,t\in\mathbb{Z}$ are such that $r\cdot s\cdot t\neq0$, then \vspace{.1cm}  
\begin{equation}{\label{htsubset}}\;\; h(\beta_{(r,s,t)})=h(r^{n_1}\cdot s^{n_2}\cdot t^{n_3})=\log\big(|r^{n_1}|\cdot |s^{n_2}|\cdot |t^{n_3}|\big)\geq\min_{1\leq i\leq3}\{n_i\}\cdot\log\max\big\{|r|,|s|,|t|\big\}. \vspace{.1cm} 
\end{equation} 
Let $U:=\subseteq \mathbb{A}^1\times\mathbb{A}^1\times\mathbb{A}^1$ be the open subset of points $P:=(r,s,t)$ such that $r\cdot s\cdot t\neq0$, and define the height function $h_{X}(P)=\log\max\big\{|r|,|s|,|t|\big\}$. Then, (\ref{htsubset}) implies that \vspace{.15cm} 
\begin{equation}{\label{bdconcise}}
h(\beta_P)\geq \min_{1\leq i\leq3}\{n_i\}\cdot h_X(P), \;\;\;\; \text{for all}\;\, P\in U(\mathbb{Z}); \vspace{.15cm} 
\end{equation}
here $U(\mathbb{Z})$ denotes the set of points of $U$ with integral coordinates. Note that the bound on (\ref{bdconcise}) will not hold on $U(\mathbb{Q})$ in general. On the other hand, it is easy to see that \vspace{.1cm}  
\begin{equation}{\label{htp7}} h(\phi_P):=h_{\mathbb{P}^7}\big([r\cdot s,0,s,t,0,1,0,1]\big)\leq2\cdot\log\max\big\{|r|,|s|,|t|\big\}=2\cdot h_X(P), \vspace{.1cm} 
\end{equation}
for all $P\in X(\mathbb{Z})$. In particular, \cite[Proposition 6(b)]{Silv-Hsia} and the bounds on (\ref{bdconcise}) and (\ref{htp7}) yield \vspace{.05cm} 
\begin{equation}{\label{lbdconical}}
\hat{h}_{\phi_P}(\beta_P)\geq \big(\min_{1\leq i\leq3}\{n_i\}-5\big)\cdot h_X(P)-C,\;\;\;\;\text{for all}\;\, P\in U(\mathbb{Z}); \vspace{.1cm}   
\end{equation} 
here $C$ is an absolute, positive constant. Hence, as in the proof of Theorem \ref{thm:families}, the upper bound in \cite[Corollary 17]{Silv-Hsia} implies that $\#(\Orb_{\phi_P}(\beta_P)\cap\mathbb{Z})\leq M(\phi,\beta)$ is bounded uniformly over all points $P\in U(\mathbb{Z})$. Therefore, it remains to control $\#(\Orb_{\phi_P}(\beta_P)\cap\mathbb{Z})$ for all $P=(r,s,t)$ such that $r\cdot s\cdot t=0$ on average. We do this in cases. 

Suppose first that $t=0$. Then one computes that $\phi_{(r,s,0)}(\beta_{(r,s,0)})=\phi_{(r,s,0)}(0)=0$; hence \vspace{.1cm}
\begin{equation}{\label{bdt}} \#\big(\Orb_{\phi_{(r,s,0)}}(\beta_{(r,s,0)}\big)\cap\mathbb{Z}\big)=1, \;\;\;\;\text{for all}\; r,s\in\mathbb{Z}. \vspace{.1cm} 
\end{equation} 
On the other hand, if $s=0$, then $\phi_{(r,0,t)}(z)=t/(z^2+1)$. Therefore, $|\phi_{(r,0,t)}(z)|\leq |t|$ is a bounded function on the real line. We deduce that, \vspace{.15cm} 
\begin{equation}\label{bds} \#\big(\Orb_{\phi_{(r,0,t)}}(\beta_{(r,0,t)}\big)\cap\mathbb{Z}\big)\leq 2B+1,\;\;\;\;\;\text{when}\;\,\max\{|r|,|t|\}\leq B. \vspace{.15cm}        
\end{equation}
Finally, suppose that $r=0$. Then $\phi_{(0,s,t)}(z)=(sz+t)/(z^2+1)$ is also a bounded function on the real line. Specifically, $|\phi_{(0,s,t)}(z)|\leq |s|+|t|$ follows from elementary calculus. Therefore, \vspace{.15cm} 
\begin{equation}\label{bdr} \#\big(\Orb_{\phi_{(0,s,t)}}(\beta_{(0,s,t)}\big)\cap\mathbb{Z}\big)\leq 4B+1,\;\;\;\;\;\text{when}\;\,\max\{|s|,|t|\}\leq B. \vspace{.15cm}       
\end{equation}
We deduce from (\ref{bdt}), (\ref{bds}), and (\ref{bdr}) that \vspace{.25cm} 
\begin{equation*}
\begin{split} 
&\frac{\sum_{|r|,|s|,|t|\leq B}\#\Big(\Orb_{\phi_{(r,s,t)}}\big(\beta_{(r,s,t)}\big)\cap\mathbb{Z}\Big)}{(2B+1)^3}=\frac{\sum_{P\in U(\mathbb{Z},B)}\#\Big(\Orb_{\phi_{P}}\big(\beta_{P}\big)\cap\mathbb{Z}\Big)}{(2B+1)^3} + \dots \\[10pt]
&+\,\,\;\;\;\;\frac{\sum_{|r|,|t|\leq B}\#\Big(\Orb_{\phi_{(r,0,t)}}\big(\beta_{(r,0,t)}\big)\cap\mathbb{Z}\Big)}{(2B+1)^3} \,+\, \frac{\sum_{|r|,|s|\leq B}\#\Big(\Orb_{\phi_{(r,s,0)}}\big(\beta_{(r,s,0)}\big)\cap\mathbb{Z}\Big)}{(2B+1)^3}\\[10pt]
&\leq\;\; \frac{\sum_{P\in U(\mathbb{Z},B)}M(\phi,\beta)}{(2B+1)^3} \,+\, \frac{\sum_{|s|,|t|\leq B}(4B+1)}{(2B+1)^3} \,+\, \frac{\sum_{|r|,|t|\leq B}(2B+1)}{(2B+1)^3} \,+\, \frac{\sum_{|r|,|s|\leq B}1}{(2B+1)^3}\\[10pt]
\end{split} 
\end{equation*}  
Letting $B$ tend to infinity, we see that 
\begin{equation*} \widebar{\Avg}_{\mathbb{Z}}(\phi,\beta,\mathbb{Z})\leq M(\phi,\beta)+3,
\end{equation*} 
and the average number of integral points  is bounded as claimed; here, $M(\phi,\beta)$ is the bound on $\#(\Orb_{\phi_{P}}\big(\beta_{P}\big)\cap\mathbb{Z})$ for all $P\in U(\mathbb{Z})$ obtained from (\ref{lbdconical}) and \cite[Corollary 17]{Silv-Hsia}.            
\end{proof} 
\indent \textbf{Acknowledgements:} It is a pleasure to thank Joseph Gunther, Bjorn Poonen, Joseph Silverman, Martin Widmer, Wei Pin Wong, and Umberto Zannier for the discussions related to the work in this paper. Finally, we thank the anonymous referee for pointing out how to use \cite[Theorem 4.1]{Call-Silverman} to get rid of some technical assumptions in an early draft of this paper.   
\end{section}  

\vspace{12mm} 
\indent\indent Wade Hindes, Department of Mathematics, CUNY Graduate Center, 365 Fifth Avenue, New York, New York 10016-4309.\\
\indent \emph{E-mail address:} \textbf{whindes@gc.cuny.edu}

\begin{thebibliography}{13}
\bibitem{Avgintell} L. Alpoge, \emph{The average number of integral points on elliptic curves is bounded}, Feb 9, 2015. Preprint, arXiv:1412.1047v3.   
\bibitem{Baker} M. Baker,  \emph{A finiteness theorem for canonical heights attached to rational maps over function fields}, Journal f\"{u}r die reine und angewandte Mathematik (Crelles Journal) 2009.626 (2009): 205-233.
\bibitem{Avgrank} M. Bhargava and A. Shankar, \emph{Binary quartic forms having bounded invariants, and the boundedness of the average rank of elliptic curves}, Annals of Math 181.1 (2015): 191-242. 
\bibitem{Call-Silverman} G. Call and J. Silverman, \emph{Canonical heights on varieties with morphisms}, Compositio Mathematica 89.2 (1993): 163-205. 
\bibitem{Silv-Hsia} L.-C Hsia and J. Silverman, \emph{A quantitative estimate for quasi-integral points in orbits.} Pacific J. Math 249.2 (2011): 321-342.
\bibitem{Heightuniformity} Su-Ion Ih, \emph{Height uniformity for algebraic points on curves}, Compositio Math, 134: 35-57, (2002).  
\bibitem{Ih2}Su-Ion Ih. \emph{Height uniformity for integral points on elliptic curves.} Trans. of the Amer. Math. Soc. 358.4 (2006): 1657-1675.
\bibitem{Kim} M. Kim, \emph{Geometric height inequalities and the Kodaira-Spencer map}. Compositio Math. 105.1: 43-54 (1997).
\bibitem{LangAlgebra} S. Lang, \emph{Algebra}, Springer-Verlag, revised third edition, New York, 2003. 
\bibitem{Lang-Ellip} S. Lang, \emph{Elliptic Curves: Diophantine Analysis}, volume 231 of Grundlehren der Mathematischen Wissenschaften. Springer-Verlag, Berlin, 1978.   
\bibitem{Lang} S. Lang, \emph{Fundamentals of Diophantine Geometry,} Springer-Verlag, New York, 1983. 
\bibitem{P-Stoll} B. Poonen and M. Stoll, \emph{Most odd degree hyperelliptic curves have only one rational point}, Ann. Math. 180.3 (2014): 1137-1166.   
\bibitem{Schanuel} S. Schanuel, \emph{Heights in number fields}, Bull. Soc. Math. France 107.4 (1979): 433-449.  
\bibitem{Schmidt} W. Schmidt, \emph{Thue's equation over function fields.} Journal of the Australian Mathematical Society. (Series A) 25.04 (1978): 385-422.
\bibitem{Serre} J-P. Serre, \emph{Topics in Galois theory}. Jones and Bartlett Publishers, Research Notes in Mathematics 1, 1992. 
\bibitem{Siegel} C. Siegel, \emph{\"{U}ber einige Anwendungen Diophantischer Approximationen}, Abh. Preuss Akad. Wiss. Phys.-Math. Kl., 1929, Nr. 1; \emph{Ges. Abh.}, Band 1, 209-266. 
\bibitem{Silv-Advanced} J. Silverman, \emph{Advanced Topics in the Arithmetic of Elliptic Curves}. Graduate Texts in Math. 151, Springer (1994).
\bibitem{Silv-Dyn} J. Silverman, \emph{The arithmetic of dynamical systems}. Vol. 241. Springer, 2007
\bibitem{Silv-Arith} J. Silverman, \emph{The arithmetic of elliptic curves}. Springer-Verlag, GTM 106, 1986. Expanded 2nd edition, 2009.
\bibitem{Silv-Int} J. Silverman, \emph{Integer points, Diophantine approximation, and iteration of rational maps.} Duke Mathematical Journal 71.3 (1993): 793-829.
\bibitem{Stoll} M. Stoll, \emph{Rational points on curves}, J. Th\'{e}or. Nombres Bordeaux 23(1): 257-277, (2011). 
\bibitem{Vojta} P. Vojta. \emph{Diophantine approximations and value distribution theory}, Lecture Notes in Mathematics, vol 1239, Springer-Verlag, New York, 1987.
\end{thebibliography}
\end{document}